\documentclass{article}
\usepackage{amssymb}
\usepackage{amsmath}

\newtheorem{theorem}{Theorem}

\newtheorem{corollary}[theorem]{Corollary}

\newtheorem{definition}[theorem]{Definition}
\newtheorem{example}[theorem]{Example}

\newtheorem{lemma}[theorem]{Lemma}

\newtheorem{proposition}[theorem]{Proposition}
\newtheorem{remark}[theorem]{Remark}

\newenvironment{proof}[1][Proof]{\noindent\textbf{#1.} }{\ \rule{0.5em}{0.5em}}

\begin{document}

\author{Cristina Flaut, Dana Piciu, Bianca Liana Bercea}
\date{}
\title{Some applications of fuzzy sets in residuated lattices}
\maketitle

\begin{abstract}
In this paper, based on ideals, we investigate residuated lattices from
fuzzy set theory and lattice theory point of view.

Ideals are important concepts in the theory of algebraic structures used for
formal fuzzy logic and \ first, we investigate the lattice of fuzzy ideals
in residuated lattices. Then we present applications of fuzzy sets in Coding
Theory and we study connections between fuzzy sets associated to ideals and
Hadamard codes.

\end{abstract}

\section{Introduction}

The notion of residuated lattice, introduced in \cite{WD} by Ward and
Dilworth, provides an algebraic framework for fuzzy logic.

Managing certain and uncertain information is a priority of artificial
intelligence, in an attempt to imitate human thinking. To make this
possible, in \cite{ZA}, Zadeh introduced the concept of    fuzzy subset of a
nonempty set.

In this paper, we study some applications of fuzzy sets in residuated
lattices.

In \cite{LIU}, this concept is applied to these algebras and the fuzzy
ideals are introduced. In Section 3, we investigate more properties of fuzzy
ideals and we study their lattice structure, which is a Heyting algebra.

In Section 4 we found connections between the fuzzy sets associated to
ideals in particular residuated lattices and Hadamard codes.

\section{Preliminaries}

A residuated lattice is an algebra $(L,\vee ,\wedge ,\odot ,\rightarrow
,0,1),$ with an order $\preceq $ \ such that

\begin{enumerate}
\item[$(i)$] $(L,\vee ,\wedge ,0,1)$ is a bounded lattice;

\item[$(ii)$] $(L,\odot ,1)$ is a commutative monoid;

\item[$(iii)$] $x\odot z\preceq y$ if and only if $x\preceq z\rightarrow y,$
for $x,y,z\in L$ , \textit{see } \cite{WD}.
\end{enumerate}

In this paper, $L$ will be denoted a residuated lattice, unless otherwise
stated.

A \emph{Heyting algebra} (\cite{BD}) is a lattice $(L,\vee ,\wedge )$ with $0
$\ such that for every $a,b\in L,$\ there exists an element $a\rightarrow
b\in L$ (called the \textit{pseudocomplement of\ }$a$\textit{\ with\ respect
to\ }$b)$\ where $a\rightarrow b=\sup \{x\in L:a\wedge x\leq b\}.$ Heyting
algebras are divisible residuated lattices.

For $x,y\in L,$ we define $x\boxplus y=x^{\ast }\rightarrow y^{\ast \ast }$
and $x\uplus y=x^{\ast }\rightarrow y,$ where $x^{\ast }=x\rightarrow 0.$ We
remark that $\boxplus $ is associative and commutative and $\uplus $ is only
associative.

We recall some rules of calculus in residuated lattices, see \cite{BPD},
\cite{T}:

\begin{enumerate}
\item[$(1)$] $x\rightarrow y=1$ if and only if $x\preceq y;$

\item[$(2)$] $x,y\preceq x\uplus y\preceq x\boxplus y,x\boxplus 0=x^{\ast
\ast },x\boxplus x^{\ast }=1,x\boxplus 1=1,x\boxplus y=y\boxplus
x,(x\boxplus y)\boxplus z=x\boxplus (y\boxplus z),$ $x\preceq y\Rightarrow
x\boxplus z\preceq y\boxplus z;$

\item[$(3)$] $x\boxplus y=(x^{\ast }\odot y^{\ast })^{\ast },$ $(x\boxplus
y)^{\ast \ast }=x\boxplus y=x^{\ast \ast }\boxplus y^{\ast \ast },$ for
every $x,y,z\in L.$
\end{enumerate}

An ideal in residuated lattices is a generalization of the similar notion \
from MV-algebras, see \cite{COM}. This concept is introduced in \cite{LIU}
using the operator $\uplus $ which is not commutative. An equivalent
definition is given in \cite{BPD} using $\boxplus .$ We remark that $%
\boxplus $ is associative and commutative and $\uplus $ is only associative.

\begin{definition}
\label{def1} (\cite{BPD}) An \emph{ideal} residuated lattice $L$ is a subset
$I\neq \emptyset $ of $L$\emph{\ }such that:

\begin{enumerate}
\item[$(i_{1})$] For $x\leq i,x\in L,$ $i\in I\Longrightarrow $ $x\in I;$

\item[$(i_{2})$] $i,j\in I$ $\Longrightarrow i\boxplus j\in I.$
\end{enumerate}
\end{definition}

Let $A$ be a non-empty set. If $[0,1]$ is the real unit interval, a fuzzzy
subset of $A$ is a function $\mu :A\longrightarrow \lbrack 0,1],$ see \cite%
{ZA}. If $\mu $ is not a constant map, then $\mu $ \ is a proper fuzzy
subset of $A$.

Let $B\subset A$ be a non-empty subset of $A.$ The map $\mu
_{B}:A\rightarrow \lbrack 0,1],$%
\begin{equation*}
\mu _{B}\left( x\right) =\{%
\begin{array}{c}
1,\text{ if }x\in B \\
0,\text{ if }x\notin B.%
\end{array}%
\end{equation*}%
(the \textit{characteristic function}) is a fuzzy subset.

The notion of fuzzy ideal in residuated lattices is introduced in \cite{LIU}
and some characterizations are obtained.

\begin{definition}
\label{def2} (\cite{LIU}) A\emph{\ fuzzy ideal of }a residuated lattice $L$
is a fuzzy subset $\mu $ of $L$ such that:

\begin{enumerate}
\item[$(fi_{1})$] $x\preceq y\Longrightarrow $ $\mu (x)\geq \mu (y);$

\item[$(fi_{2})$] $\mu (x\uplus y)\geq \min (\mu (x),\mu (y)),$ for every $%
x,y\in L.$
\end{enumerate}
\end{definition}

Two equivalent definitions for fuzzy ideals are given in \cite{LIU}:

A \emph{fuzzy ideal of }$L$ is a fuzzy subset $\mu $ of $L$ such that:

\begin{enumerate}
\item[$(fi_{3})$] $\mu (0)\geq \mu (x),$ for every $x\in L;$

\item[$(fi_{4})$] $\mu (y)\geq \min (\mu (x),\mu ((x^{\ast }\rightarrow
y^{\ast })^{\ast }),$ for every $x,y\in L\Leftrightarrow (fi_{4}^{\prime })$
$\mu (y)\geq \min (\mu (x),\mu (x^{\ast }\odot y)),$ for every $x,y\in L$.
\end{enumerate}

We denote by $\mathcal{I}(L)$ the set of ideals and by $\mathcal{FI}(L)$ the
set of fuzzy ideals of the residuated lattice $L.$

Obviously, the constant functions\textbf{\ }$\mathbf{0,1}:L\rightarrow
\lbrack 0,1],$ $\mathbf{0}(x)=0$ and $\mathbf{1}(x)=1,$ for every $x\in L$
are fuzzy ideals of $L.$

There are two important fuzzy subsets in a residuated lattice $L:$ \ For $%
I\subseteq L$ and $\alpha ,\beta \in \lbrack 0,1]$ with $\alpha >\beta $ is
defined $\mu _{I}:L\rightarrow \lbrack 0,1]$ by
\begin{equation*}
\mu _{I}(x)=\{%
\begin{array}{c}
\alpha ,\text{ if }x\in I \\
\beta ,\text{ if }x\notin I.%
\end{array}%
\end{equation*}%
The fuzzy subset $\mu _{I}$ is a generalization of the characteristic
function of $I,$ denoted $\varphi _{I.}$ Moreover, in \cite{LIU} is proved
that $I\in \mathcal{I}(L)$ iff $\mu _{I}\in \mathcal{F}\mathcal{I}(L).$

\begin{lemma}
\label{lema1} (\cite{LIU}) \ For $\mu \in \mathcal{FI}(L),$ the following
hold:

\begin{enumerate}
\item[$(i)$] $\mu (x)=\mu (x^{\ast \ast })$

\item[$(ii)$] $\mu (x\uplus y)=\min (\mu (x),\mu (y)),$ for every $x,y\in L.$
\end{enumerate}
\end{lemma}

For $\mu _{1}$ and $\mu _{2}$ two fuzzy subsets of $L$ \ is define the order
relation $\mu _{1}\subset $ $\mu _{2}$ if $\mu _{1}(x)\leq $ $\mu _{2}(x),$
for every $x\in L.$

Moreover, for a family $\{\mu _{i}:i\in I\}$ of fuzzy ideals of $L$ we
define $\underset{i\in I}{\cup }\mu _{i},\underset{i\in I}{\cap }\mu
_{i}:L\rightarrow \lbrack 0,1]$ by
\begin{equation*}
(\underset{i\in I}{\cup }\mu _{i})(x)=\sup \{\mu _{i}(x):i\in I\}\text{ and }%
(\underset{i\in I}{\cap }\mu _{i})(x)=\inf \{\mu _{i}(x):i\in I\},\text{ for
every }x\in L,\text{ see }\cite{ZA}.
\end{equation*}

Obviously, $\underset{i\in I}{\cap }\mu _{i}\in \mathcal{FI}(L)$ but, in
general $\underset{i\in I}{\cup }\mu _{i}$ is not a fuzzy ideal of $L,$ see
\cite{LELE}.

We recall (see \cite{BD}) that a complete lattice $(\mathcal{A},\vee ,\wedge
)$ is called Brouwerian if it satisfies the identity $a\wedge (\underset{i}{%
\bigvee }b_{i})=\underset{i}{\bigvee }(a\wedge b_{i}),$ whenever the
arbitrary unions exists. An element $a\in \mathcal{A}$ is called compact if $%
a\leq \vee X$ \ for some $X\subseteq \mathcal{L}$ implies $a\leq \vee X_{1}$
for some finite $X_{1}\subseteq X.$

\begin{remark}
(\cite{BD}) Let $A$ be a set \ of real numbers. We say that $l\in R$ is the
supremum of $A$ if:

\begin{enumerate}
\item $l$ is an upper bound for $A;$

\item $l$ is the least upper bound: for every $\ \epsilon >0$ there is $%
a_{\epsilon }\in A$ such that $a_{\epsilon }>l-\epsilon ,$ i.e., $%
l<a_{\epsilon }+\epsilon .$
\end{enumerate}
\end{remark}

\begin{remark}
\label{rem0} If $a,b$ are real numbers such that $a,b\in \lbrack 0,1]$ and $%
a>b-\epsilon ,$ for every $\epsilon >0,$ then $a\geq b.$ Indeed, if we
suppose that $a<b,$ then there is $\epsilon _{0}>0$ such that $b-a>\epsilon
_{0}>0,$ which is a contradiction with hypothesis.
\end{remark}

\section{The lattice of fuzzy ideals in a residuated lattice $L$}

\begin{lemma}
\label{lema0} Let $x,y,z\in L.$ Then $x^{\ast }\boxplus (y\boxplus z)=1$ iff
$x\preceq y\boxplus z.$
\end{lemma}

\begin{proof}
If $x^{\ast }\boxplus (y\boxplus z)=1,$ then $1=x^{\ast \ast }\rightarrow
(y\boxplus z)^{\ast \ast }=x^{\ast \ast }\rightarrow (y\boxplus z),$ so $%
x\preceq x^{\ast \ast }\preceq y\boxplus z.$

Conversely, $x\preceq y\boxplus z\Rightarrow $ $x^{\ast \ast }\preceq
(y\boxplus z)^{\ast \ast }\Rightarrow x^{\ast \ast }\rightarrow (y\boxplus
z)^{\ast \ast }=1\Rightarrow x^{\ast }\boxplus (y\boxplus z)=1.$
\end{proof}

\begin{lemma}
\label{lema2} If $\mu \in \mathcal{FI}(L),$ then $\mu (x\boxplus y)=\mu
(x\uplus y)=\min (\mu (x),\mu (y^{\ast \ast })),$ for every $x,y\in L.$
\end{lemma}

\begin{proof}
From Lemma \ref{lema1}, $\mu (x\boxplus y)=\mu (x\uplus y^{\ast \ast })=\min
(\mu (x),\mu (y^{\ast \ast }))=\min (\mu (x),\mu (y))=\mu (x\uplus y).$
\end{proof}

\begin{proposition}
\label{prop1} Let $\mu $ be a fuzzy subset of $L.$ Then $\mu \in \mathcal{FI}%
(L)$ iff it satisfies the following conditions:

\begin{enumerate}
\item[$(fi_{1})$] $x\preceq y\Longrightarrow $ $\mu (x)\geq \mu (y);$

\item[$(fi_{2}^{\prime })$] $\mu (x\boxplus y)\geq \min (\mu (x),\mu (y)),$
for every $x,y\in L.$
\end{enumerate}
\end{proposition}

\begin{proof}
If $\mu \in \mathcal{FI}(L),$ using Lemma \ref{lema2}, $(fi_{2}^{\prime })$
holds.

Conversely, assume that $(fi_{1})$ and $(fi_{2}^{\prime })$ hold and let $%
x,y\in L.$ Since $x\uplus y\preceq x\boxplus y$ we obtain $\min (\mu (x),\mu
(y))\leq \mu (x\boxplus y)\leq \mu (x\uplus y),$ so $(fi_{2})$ hold. Thus, $%
\mu \in \mathcal{FI}(L).$
\end{proof}

\begin{proposition}
\label{prop2} Let $\mu $ be a fuzzy subset of $L.$ The following are
equivalent:

\begin{enumerate}
\item[$(i)$] $\mu \in \mathcal{FI}(L);$

\item[$(ii)$] For every $x,y,z\in L,$ if $(x\boxplus y)\boxplus z^{\ast }=1$
then $\mu (z)\geq \min (\mu (x),\mu (y));$

\item[$(iii)$] For every $x,y,z\in L,$ if $z\preceq x\boxplus y$ then $\mu
(z)\geq \min (\mu (x),\mu (y)).$
\end{enumerate}
\end{proposition}

\begin{proof}
$(i)\Longrightarrow (ii).$ Let $x,y,z\in L$ such that $(x\boxplus y)\boxplus
z^{\ast }=1.$ Then $1=(x\boxplus y)^{\ast }\longrightarrow z^{\ast }$ so, $%
(x\boxplus y)^{\ast }\preceq z^{\ast }$. Thus, using Lemma \ref{lema1} and
Proposition \ref{prop1} we have $\mu (z)=\mu (z^{\ast \ast })\geq \mu
((x\boxplus y)^{\ast \ast })=\mu (x\boxplus y)\geq \min (\mu (x),\mu (y)).$

$(ii)\Longrightarrow (i).$ Since $(x\boxplus x)\boxplus 0^{\ast }=1,$ by
hypothesis, we deduce $(fi_{3}).$ Also, since $[x\boxplus (x^{\ast }\odot
y)]\boxplus y^{\ast }=$ $(x\boxplus y^{\ast })\boxplus (x^{\ast }\odot
y)=(x^{\ast }\odot y)^{\ast }\boxplus (x^{\ast }\odot y)=1,$ we obtain $%
(fi_{4}^{\prime }).$Thus, $\mu \in \mathcal{FI}(L).$

$(ii)\Leftrightarrow (iii).$ Using Lemma \ref{lema0}, $z\preceq x\boxplus y$
iff $(x\boxplus y)\boxplus z^{\ast }=1.$
\end{proof}

If $\mu $ is a fuzzy subset of a residuated lattice $L,$ we denote by $%
\overline{\mu }$ the smallest fuzzy ideal containing $\mu .$ $\overline{\mu }
$ \ is called the fuzzy ideal generated by $\mu $ and it is characterized in
\cite{LIU}, Theorem 3.19 and \cite{LELE}, Theorem 5.

In the following, we show a new characterization:

\begin{proposition}
\label{prop3} Let $L$ be a residuated lattice and $\mu ,\mu ^{\prime }$ $%
:L\rightarrow \lbrack 0,1]$ be fuzzy subsets of $L$ such that%
\begin{equation*}
\mu ^{\prime }(x)=\sup \{\min (\mu (x_{1}),...,\mu (x_{n})):x\preceq
x_{1}\boxplus ...\boxplus x_{n},n\in N,x_{1},....,x_{n}\in L\},\text{ for
every }x\in L.\text{ }
\end{equation*}%
Then $\mu ^{\prime }=\overline{\mu }.$
\end{proposition}

\begin{proof}
First, using Proposition \ref{prop2}, we will prove that $\mu ^{\prime }\in
\mathcal{FI}(L).$

Let $x,y,z\in L$ such that $z\preceq x\boxplus y$ and $\epsilon >0$
arbitrary.

By definition of $\mu ^{\prime }$, for $x,y\in L$ there are $n,m\in N$ and $%
x_{1},....,x_{n},y_{1},....,y_{m}\in L$ such that
\begin{equation*}
x\preceq x_{1}\boxplus ...\boxplus x_{n}\text{ and }\mu ^{\prime
}(x)<\epsilon +\min (\mu (x_{1}),...,\mu (x_{n}))\text{ }
\end{equation*}

and%
\begin{equation*}
y\preceq y_{1}\boxplus ...\boxplus y_{m}\text{ and }\mu ^{\prime
}(y)<\epsilon +\min (\mu (y_{1}),...,\mu (y_{m})).\text{ }
\end{equation*}

Then $x\boxplus y\preceq x_{1}\boxplus ...\boxplus x_{n}\boxplus
y_{1}\boxplus ...\boxplus y_{m}$ and $\mu ^{\prime }(x\boxplus y)=\sup
\{\min (\mu (t_{1}),...,\mu (t_{k})):x\boxplus y\preceq t_{1}\boxplus
...\boxplus t_{k},k\in N,t_{1},....,t_{k}\in L\}$ $\geq \min (\mu
(x_{1}),...,\mu (x_{n}),\mu (y_{1}),...,\mu (y_{m}))$ $=\min (\min (\mu
(x_{1}),...,\mu (x_{n})),\min (\mu (y_{1}),...,\mu (y_{m})))$ $>\min (\mu
^{\prime }(x)-\epsilon ,\mu ^{\prime }(y)-\epsilon )=\min (\mu ^{\prime
}(x),\mu ^{\prime }(y))-\epsilon .$

Since $\epsilon $ is arbitrary, using Remark \ref{rem0}, we deduce that $\mu
^{\prime }(x\boxplus y)\geq \min (\mu ^{\prime }(x),\mu ^{\prime }(y)).$

Similarly, for $x\boxplus y$ there are $p\in N$ and $s_{1},....,s_{p}\in L$
such that
\begin{equation*}
x\boxplus y\preceq s_{1}\boxplus ...\boxplus s_{p}\text{ and }\mu ^{\prime
}(x\boxplus y)<\epsilon +\min (\mu (s_{1}),...,\mu (s_{p})).\text{ }
\end{equation*}

Thus, $z\preceq s_{1}\boxplus ...\boxplus s_{p},$ so $\mu ^{\prime }(z)=\sup
\{\min (\mu (z_{1}),...,\mu (z_{r})):z\preceq z_{1}\boxplus ...\boxplus
z_{r},r\in N,z_{1},....,z_{r}\in L\}$ $\geq \min (\mu (s_{1}),...,\mu
(s_{p}))$ $>\mu ^{\prime }(x\boxplus y)-\epsilon .$

We obtain $\mu ^{\prime }(z)\geq \mu ^{\prime }(x\boxplus y).$ Finally, we
conclude that $\mu ^{\prime }(z)\geq \min (\mu ^{\prime }(x),\mu ^{\prime
}(y)),$ so $\mu ^{\prime }\in \mathcal{FI}(L).$

Obviously, $\mu \subset \mu ^{\prime }$ since for every $x\in L,x\preceq
x\boxplus x,$ so $\mu ^{\prime }(x)\geq \min (\mu (x),\mu (x))=\mu (x).$

Also, if $\mu ^{^{\prime \prime }}\in \mathcal{FI}(L)$ such that $\mu
\subset \mu ^{^{\prime \prime }\text{ }}$then $\mu ^{\prime }(x)=\sup \{\min
(\mu (x_{1}),...,\mu (x_{n})):x\preceq x_{1}\boxplus ...\boxplus x_{n},n\in
N,x_{1},....,x_{n}\in L\}\leq $ $\sup \{\min (\mu ^{\prime \prime
}(x_{1}),...,\mu ^{\prime \prime }(x_{n})):x\preceq x_{1}\boxplus
...\boxplus x_{n},n\in N,x_{1},....,x_{n}\in L\}\leq \mu ^{\prime \prime
}(x),$ for every $x\in L,$ since $x\preceq x_{1}\boxplus ...\boxplus
x_{n}\Rightarrow $ $\mu ^{\prime \prime }(x)\geq \mu ^{\prime \prime
}(x_{1}\boxplus ...\boxplus x_{n})=\min (\mu ^{\prime \prime
}(x_{1}),...,\mu ^{\prime \prime }(x_{n})).$Thus, $\mu ^{\prime }\subset \mu
^{^{\prime \prime }\text{ }}$ , so $\mu ^{\prime }$ is the least fuzzy ideal
of $L$ containing $\mu ,$ i.e., $\mu ^{\prime }=\overline{\mu }.$
\end{proof}

\begin{theorem}
\label{th1}The lattice $(\mathcal{FI}(L),\subset $ $)$ is a complete
Brouwerian lattice.
\end{theorem}

\begin{proof}
\textbf{\ }If $(\mu _{i})_{i\in I}$ is a family of fuzzy ideals of $L,$ then
the infimum of this family is $\underset{i\in I}{\sqcap }\mu _{i}=\underset{%
i\in I}{\cap }\mu _{i}$ and the supremum is $\underset{i\in I}{\sqcup }\mu
_{i}=\overline{\underset{i\in I}{\cup }\mu _{i}}.$

Obviously, the lattice $(\mathcal{FI}(L),\subset $ $)$ is complete.

To prove that $\mathcal{FI}(L)$ is a Brouwerian lattice we show that for
every fuzzy ideal $\mu $ and every family $(\mu _{i})_{i\in I}$ of fuzzy
ideals, $\mu \sqcap (\underset{i\in I}{\sqcup }\mu _{i})=\underset{i\in I}{%
\sqcup }(\mu \sqcap \mu _{i}).$ Clearly, $\underset{i\in I}{\sqcup }(\mu
\sqcap \mu _{i})\subset \mu \sqcap (\underset{i\in I}{\sqcup }\mu _{i}),$ so
we prove only that $\mu \sqcap (\underset{i\in I}{\sqcup }\mu _{i})\subset
\underset{i\in I}{\sqcup }(\mu \sqcap \mu _{i}).$

For this, let $x\in L$ and $\epsilon >0$ arbitrary.

Since $(\underset{i\in I}{\sqcup }\mu _{i})(x)=\sup \{\min ((\underset{i\in I%
}{\cup }\mu _{i})(z_{1}),...,(\underset{i\in I}{\cup }\mu
_{i})(z_{m})):x\preceq z_{1}\boxplus ...\boxplus z_{m},m\in
N,z_{1},....,z_{m}\in L\},$ there are $n\in N$ and $x_{1},....,x_{n}\in L$
such that
\begin{equation*}
x\preceq x_{1}\boxplus ...\boxplus x_{n}\text{ and }(\underset{i\in I}{%
\sqcup }\mu _{i})(x)<\epsilon +\min ((\underset{i\in I}{\cup }\mu
_{i})(x_{1}),...,(\underset{i\in I}{\cup }\mu _{i})(x_{n})).
\end{equation*}

Using the definition of $\underset{i\in I}{\cup }\mu _{i},$ for every $%
k=1,...,n$ there is $i_{k}\in N$ such that
\begin{equation*}
\text{ }(\underset{i\in I}{\cup }\mu _{i})(x_{k})<\epsilon +\mu
_{i_{k}}(x_{k}).
\end{equation*}

Thus,
\begin{equation*}
\text{ }(\underset{i\in I}{\sqcup }\mu _{i})(x)<\epsilon +\min (\epsilon
+\mu _{i_{1}}(x_{1}),...,\epsilon +\mu _{i_{n}}(x_{n})).
\end{equation*}

Then
\begin{equation*}
(\mu \sqcap (\underset{i\in I}{\sqcup }\mu _{i}))(x)<2\epsilon +\min (\mu
(x),\mu _{i_{1}}(x_{1}),...,\mu _{i_{n}}(x_{n})).
\end{equation*}

We consider $y_{1},....,y_{n}\in L$ such that
\begin{equation*}
y_{1}^{\ast }=(y_{2}\boxplus ...\boxplus y_{n})\boxplus x^{\ast }
\end{equation*}%
\begin{equation*}
y_{n}^{\ast }=(x_{1}\boxplus ...\boxplus x_{n-1})\boxplus x^{\ast }
\end{equation*}

and for every $t=2,...,n-1$
\begin{equation*}
y_{t}^{\ast }=(x_{1}\boxplus ...\boxplus x_{t-1})\boxplus (y_{t+1}\boxplus
...\boxplus y_{n})\boxplus x^{\ast }.
\end{equation*}

Obviously, for every $t=1,...,n,$ $y_{t}^{\ast }\boxplus x=1,$ so, $%
y_{t}^{\ast \ast }\preceq x^{\ast \ast }$ and $\mu (x)=\mu (x^{\ast \ast
})\leq \mu (y_{t}^{\ast \ast })=\mu (y_{t}).$

Moreover, $(y_{1}\boxplus ...\boxplus y_{n})\boxplus x^{\ast }=y_{1}\boxplus
y_{1}^{\ast }=1,$ so using Lemma \ref{lema0}, we deduce that%
\begin{equation*}
x\preceq y_{1}\boxplus ...\boxplus y_{n}.
\end{equation*}

Also, by Lemma \ref{lema0}, since $x\preceq x_{1}\boxplus ...\boxplus x_{n}$
we have that $y_{n}^{\ast }\boxplus x_{n}=(x_{1}\boxplus ...\boxplus
x_{n})\boxplus x^{\ast }=1$ and for every $t=1,...,n-1,$ $y_{t}^{\ast
}\boxplus x_{t}=[(x_{1}\boxplus ...\boxplus x_{t})\boxplus $ $%
(y_{t+2}\boxplus ...\boxplus y_{n}\boxplus x^{\ast }]\boxplus
y_{t+1}=y_{t+1}^{\ast }\boxplus y_{t+1}=1.$

So,%
\begin{equation*}
y_{t}\preceq x_{t},\text{ for every }t=1,...,n.
\end{equation*}

Thus, we deduce that
\begin{equation*}
\mu _{i_{k}}(x_{k})\leq \mu _{i_{k}}(y_{k}),\text{ for every }k=1,...,n.
\end{equation*}

We conclude that
\begin{equation*}
\min (\mu (x),\mu _{i_{k}}(x_{k}))\leq \min (\mu (y_{k}),\mu
_{i_{k}}(y_{k}))=(\mu \sqcap \mu _{i_{k}})(y_{k}),\text{ for every }%
k=1,...,n.
\end{equation*}

Thus,
\begin{equation*}
(\mu \sqcap (\underset{i\in I}{\sqcup }\mu _{i}))(x)<2\epsilon +\min ((\mu
\sqcap \mu _{i_{1}})(y_{1}),...(\mu \sqcap \mu _{i_{n}})(y_{n})).
\end{equation*}

Since $(\mu \sqcap \mu _{i_{k}})(y_{k})\leq (\underset{i\in I}{\sqcup }(\mu
\sqcap \mu _{i}))(y_{k}),$ for every $k=1,...,n,$ using the fact that $%
x\preceq y_{1}\boxplus ...\boxplus y_{n},$ we obtain
\begin{equation*}
(\mu \sqcap (\underset{i\in I}{\sqcup }\mu _{i}))(x)<2\epsilon +\min ((%
\underset{i\in I}{\sqcup }(\mu \sqcap \mu _{i}))(y_{1}),...,(\underset{i\in I%
}{\sqcup }(\mu \sqcap \mu _{i}))(y_{n}))<2\epsilon +(\underset{i\in I}{%
\sqcup }(\mu \sqcap \mu _{i}))(x).
\end{equation*}

But $\epsilon $ is arbitrary, so from Remark \ref{rem0},
\begin{equation*}
(\mu \sqcap (\underset{i\in I}{\sqcup }\mu _{i}))(x)\leq (\underset{i\in I}{%
\sqcup }(\mu \sqcap \mu _{i}))(x).
\end{equation*}
\end{proof}

By \cite{BD} and Theorem \ref{th1} we deduce that:

\begin{proposition}
\label{prop4}If $\mu _{1},$ $\mu _{2}\in \mathcal{FI}(L)$ then

\begin{enumerate}
\item[$(i)$] $\mu _{1}\rightsquigarrow $ $\mu _{2}=\sup \{\mu \in \mathcal{FI%
}(L):\mu _{1}\sqcap \mu \subset \mu _{2}\}=\sqcup \{\mu \in \mathcal{FI}%
(L):\mu _{1}\sqcap \mu \subset \mu _{2}\}\in \mathcal{FI}(L);$

\item[$(ii)$] If $\mu \in \mathcal{FI}(L),$ then $\mu _{1}\sqcap \mu \subset
\mu _{2}$ if and only if $\mu \sqsubset \mu _{1}\rightsquigarrow \mu _{2}.$
\end{enumerate}
\end{proposition}

Moreover,

\begin{corollary}
\label{cor1} $(\mathcal{FI}(L),\sqcap ,\sqcup ,\rightsquigarrow ,\mathbf{0})$
is a Heyting algebra.
\end{corollary}

\section{\textbf{Applications of fuzzy sets in Coding Theory}}

\subsection{\textbf{\ Symmetric difference of ideals in a finite commutative
and unitary ring}}

In this section, we will present an application of a fuzzy sets on some
special cases of residuated algebras, namely Boolean algebras.

Let $A$ be a non-empty set and $B\subset A$ be a non-empty subset of $A.$
The map $\mu _{B}:A\rightarrow \lbrack 0,1],$%
\begin{equation*}
\mu _{B}\left( x\right) =\left\{
\begin{array}{c}
{1,~x\in B} \\
{0,~x\notin B}%
\end{array}%
\right. ,
\end{equation*}%
is called the \textit{characteristic function} of the set $B$.

For two nonempty sets, $A,B,$ we define the \textit{symmetric difference }of
the sets $A,B,$ \thinspace
\begin{equation*}
A\Delta B=\left( A-B\right) \cup \left( B-A\right) =\left( A\cup B\right)
-\left( B\cap A\right)
\end{equation*}

\begin{proposition}
\bigskip \label{prop3.1} \textit{We consider} $A$ \textit{and} $B$ \textit{%
two nonempty sets.}

\begin{enumerate}
\item[$(i)$] \textit{We have }$\mu _{A\Delta B}=0$ \textit{if and only if} $%
A=B;$

\item[$(ii)$] (\cite{KP}, p. 215). \textit{The following relation is true}
\begin{equation*}
\mu _{A\Delta B}=\mu _{A}+\mu _{B}-2\mu _{A}\mu _{B}.
\end{equation*}

\item[$(iii)$] \textit{Let} $A_{i},i\in \{1,2,...,n\}$ \textit{be} $n$
\textit{nonempty sets. The following relation is true}%
\begin{equation*}
\mu _{A_{1}\Delta A_{2}\Delta ...\Delta A_{n}}=\sum \mu _{A_{i}}-2\underset{%
i\neq j}{\sum }\mu _{A_{i}}\mu _{A_{j}}+2^{2}\underset{i\neq j\neq k}{\sum }%
\mu _{A_{i}}\mu _{A_{j}}\mu _{A_{k}}-...+\left( -1\right) ^{n-1}2^{n-1}\mu
_{A_{1}}\mu _{A_{2}}...\mu _{A_{n}}.
\end{equation*}%
\textit{\ }
\end{enumerate}
\end{proposition}

\begin{remark}
\label{rem3.2}  Let $\left( R,+,\cdot \right) $ be a unitary and a
commutative ring and $I_{1},I_{2},...,I_{s}$ be ideals in $R$.

\begin{enumerate}
\item[$(i)$] For $i\neq j,$ we have $I_{i}\Delta I_{j}$ is not an ideal in $R
$. Indeed, $0\notin I_{i}\Delta I_{j}$, therefore $I_{i}\Delta I_{j}$ is not
an ideal in $R$;

\item[$(ii)$] In general, $I_{1}\Delta I_{2}\Delta ...\Delta I_{n}$, for $%
n\geq 2$, is not an ideal in $R$. Indeed, if $n\geq 3$ and $x,y\in
I_{1}\Delta I_{2}\Delta ...\Delta I_{n},$ supposing that $x\in I_{j}$ and $%
y\in I_{k},$ we have that $xy\in I_{j}$ and $xy\in I_{k},$ therefore $xy\in
I_{j}\cap I_{k}$. We obtain that $\mu _{I_{1}\Delta I_{2}\Delta ...\Delta
I_{n}}\left( xy\right) =\mu _{I_{j}}\left( xy\right) +\mu _{I_{k}}\left(
xy\right) -2\mu _{_{I_{j}}}\mu _{_{I_{k}}}\left( xy\right) =0,$ then $%
xy\notin I_{1}\Delta I_{2}\Delta ...\Delta I_{n}$ and $I_{1}\Delta
I_{2}\Delta ...\Delta I_{n}$ is not an ideal in $R$.\medskip
\end{enumerate}
\end{remark}

\begin{definition}
\label{def3.3} If $A=\{a_{1},a_{2},...,a_{n}\}$ is a finite set with $n$
elements and \thinspace $B\,$\ is a nonempty subset of $A$, we consider the
vector $c_{B}=(c_{i})_{i\in \{1,2,...,n\}},$ where $c_{i}=0$ if $a_{i}\notin
B$ and $c_{i}=1$ if $a_{i}\in B$. The vector $c_{B}$ is called the \textit{%
codeword} attached to the set $B$. We can represent $c_{B}$ as a string $%
c_{B}=c_{1}c_{2}...c_{n}$.
\end{definition}

\subsection{\textbf{Linear codes}}

We consider $p$ a prime number and $\mathbf{F}_{p^{n}}$ a finite field of
characteristic $p$. $\mathbf{F}_{p^{n}}$ is a vector space over the field $%
\mathbb{Z}_{p}.$ A \textit{linear code} $\mathcal{C}$ of length $n$ and
dimension $k$ is a vector subspace of the vector space $\mathbf{F}_{p^{n}}$.
If $p=2$, we call this code a binary linear code. The elements of $\mathcal{C%
}$ are called \textit{codewords}. The \textit{weight} of a codeword is the
number of its elements that are nonzero and the \textit{distance} between
two codewords is the \textit{Hamming distance} between them, that means
represents the number of elements in which they differ. The distance $d$ of
the linear code is the minimum weight of its nonzero codewords, or
equivalently, the minimum distance between distinct codewords. A linear code
of length $n$, dimension $k$, and distance $d$ is called an $[n,k,d]$ code
(or, more precisely, $[n,k,d]_{p}$ code). \textit{The rate} of a code is $%
\frac{k}{n},$ that means it is an amount such that for each $k$ bits of
transmitted information, the code generates $n$ bits of data, in which $n-k$
are redundant. Since $\mathcal{C}$ is a vector subspace of dimension $k$, it
is generated by bases of $k$ vectors. The elements of such a basis can be
represented as a rows of a matrix $G$, named \textit{generating matrix}
associated to the code $\mathcal{C}$. This matrix is a matrix of $k\times n$
type. (see [Gu; 10]). The codes of the type $[2^{t},t,2^{\tau -1}]_{2}$, $%
t\geq 2$, are called \textit{Hadamard codes}. Hadamard codes are a class of
error-correcting codes (see [KK; 12], p. 183). Named after french
mathematician Jacques Hadamard, these codes are used for error detection and
correction when transmitting messages are over noisy or unreliable channels.
Usually, Hadamard codes are constructed by using Hadamard matrices of
Sylvester's type, but there are Hadamard codes using arbitrary Hadamard
matrix not necessarily of the above type (see [CR; 20]). As we can see,
Hadamard codes have a good distance property, but the rate is of a low level
(see [Gu; 10]).

\begin{remark}
\label{rem3.4} (\cite{Gu}, Definition 16). The generating matrix of a
Hadamard code of the type $[2^{t},t,2^{\tau -1}]_{2}$, $t\geq 2,$ has as
columns all $t$-bits vectors over $\mathbb{Z}_{2}$ (vectors of length $t$).
\end{remark}

\section{\textbf{\ Connections between Boolean algebras and Hadamard codes}}

In the following, we present a particular case of residuated lattices, named
\textit{MV-algebras.}\medskip

\begin{definition}
\label{def3.5} (\cite{CHA}) An abelian monoid $\left( X,\theta ,\oplus
\right) $ is called \textit{MV-algebra} if and only if we have an operation $%
"^{\prime }"$ such that:

\begin{enumerate}
\item[$(i)$] $(x^{\prime })^{\prime }=x;$

\item[$(ii)$] $x\oplus \theta ^{\prime }=\theta ^{\prime };$

\item[$(iii)$] $\left( x^{\prime }\oplus y\right) ^{\prime }\oplus y=$ $%
\left( y^{\prime }\oplus x\right) ^{\prime }\oplus x$, for all $x,y\in X.$
We denote it by $\left( X,\oplus ,^{\prime },\theta \right) .\medskip $
\end{enumerate}
\end{definition}

\begin{definition}
\label{def3.6}\textbf{\ }(\cite{COM}, Definition 4.2.1) An algebra $\left(
W,\circ ,\overline{\phantom{x}},1\right) $ of type $\left( 2,1,0\right) ~$is
called a \textit{Wajsberg algebra (}or\textit{\ W-algebra)} if and only if
for every $x,y,z\in W$, we have:

\begin{enumerate}
\item[$(i)$] $1\circ x=x;$

\item[$(ii)$] $\left( x\circ y\right) \circ \left[ \left( y\circ z\right)
\circ \left( x\circ z\right) \right] =1;$

\item[$(iii)$]  $\left( x\circ y\right) \circ y=\left( y\circ x\right) \circ
x;$

\item[$(iv)$] $\left( \overline{x}\circ \overline{y}\right) \circ \left(
y\circ x\right) =1.\medskip $
\end{enumerate}
\end{definition}

\begin{remark}
\label{rem3.7}\textbf{\ }(\cite{COM}, Lemma 4.2.2 and Theorem 4.2.5)

\begin{enumerate}
\item[$(i)$] If $\left( W,\circ ,\overline{\phantom{x}},1\right) $ is a
Wajsberg algebra, defining the following multiplications
\begin{equation*}
x\odot y=\overline{\left( x\circ \overline{y}\right) }
\end{equation*}%
and
\begin{equation*}
x\oplus y=\overline{x}\circ y,
\end{equation*}%
for all $x,y\in W$, we obtain that $\left( W,\oplus ,\odot ,\overline{%
\phantom{x}},0,1\right) $ is an MV-algebra.

\item[$(ii)$] If $\left( X,\oplus ,\odot ,^{\prime },\theta ,1\right) $ is
an MV-algebra, defining on $X$ the operation%
\begin{equation*}
x\circ y=x^{\prime }\oplus y,
\end{equation*}%
it results that $\left( X,\circ ,^{\prime },1\right) $ is a Wajsberg
algebra.\medskip
\end{enumerate}
\end{remark}

\begin{definition}
\label{def3.8} (\cite{FRT}) If $\left( W,\circ ,\overline{\phantom{x}}%
,1\right) $ is a Wajsberg algebra, on $W$ we define the following binary
relation
\begin{equation}
x\leq y~\text{if~and~only~if~}x\circ y=1.  \tag{3.2.}
\end{equation}%
This relation is an order relation, called \textit{the natural order
relation on }$W$\textit{.\medskip }
\end{definition}

\begin{definition}
\label{def3.9}(\cite{CT}) Let $\left( X,\oplus ,^{\prime },\theta \right) $
be an MV-algebra. The nonempty subset $I\subseteq X$ is called an \textit{%
ideal} in $X$ if and only if the following conditions are satisfied:

\begin{enumerate}
\item[$(i)$] $\theta \in I$, where $\theta =\overline{1};$

\item[$(ii)$] $x\in I$ and $y\leq x$ implies $y\in I;$

\item[$(iii)$] If $x,y\in I$, then $x\oplus y\in I$.\medskip
\end{enumerate}
\end{definition}

We remark that the concept of ideal in residuated lattices is a
generalization for the notion of ideal in MV-algebras.

\begin{definition}
\label{def3.10}\textbf{\ }(\cite{COM}, p. 13) An ideal $P$ of the MV-algebra
$\left( X,\oplus ,^{\prime },\theta \right) $ is a \textit{prime} ideal in $X
$ if and only if for all $x,y\in P$ we have $(x^{\prime }\oplus y)^{\prime
}\in P$ or $(y^{\prime }\oplus x)^{\prime }\in P$.\medskip
\end{definition}

\begin{definition}
\label{def3.11}\textbf{\ (\cite{GA}}, p. 56) Let $\left( W,\circ ,\overline{%
\phantom{x}},1\right) $ be a Wajsberg algebra and let $I\subseteq W$ be a
nonempty subset. $I$ is called an \textit{ideal} in $W$ if and only if the
following conditions are fulfilled:

\begin{definition}
\begin{enumerate}
\item[$(i)$]  $\theta \in I$, where $\theta =\overline{1};$

\item[$(ii)$]  $x\in I$ and $y\leq x$ implies $y\in I;$

\item[$(iii)$]  If $x,y\in I$, then $\overline{x}\circ y\in I$.
\end{enumerate}
\end{definition}
\end{definition}

\begin{definition}
\label{def3.12} Let $\left( W,\circ ,\overline{\phantom{x}},1\right) $ be a
Wajsberg algebra and $P\subseteq W$ be a nonempty subset. $P$ is called a
\textit{prime} \textit{ideal} in $W$ if and only if for all $x,y\in P$ we
have $(x\circ y)^{\prime }\in P$ or $(y\circ x)^{\prime }\in P$.\medskip
\end{definition}

\begin{definition}
\label{def3.13}The algebra $\left( B,\vee \wedge ,\partial ,0,1\right) $,
equipped with two binary operations $\vee $ and $\wedge $ and a unary
operation$\ \partial $,$~$is called a \textit{Boolean algebra} if and only
if $\left( B,\vee \wedge \right) ~$is a distributive and a complemented
lattice with
\begin{equation*}
x\vee \partial x=1,
\end{equation*}%
\begin{equation*}
x\wedge \partial x=0,
\end{equation*}%
for all elements $x\in B$. The elements $0$ and $1$ are the least and the
greatest elements from the algebra $B.\medskip $
\end{definition}

\begin{remark}
\label{rem3.14}

\begin{enumerate}
\item[$(i)$] Boolean algebras represent a particular case of MV-algebras.
Indeed, if $\left( B,\vee \wedge ,\partial ,0,1\right) $ is a Boolean
algebra, then can be easily checked that $\left( B,\vee ,\partial ,0\right) $
is an MV-algebra;

\item[$(ii)$] A \textit{Boolean ring} $\left( B,+,\cdot \right) $ is a
unitary and commutative ring such that $x^{2}=x,$ for each $x\in B;$

\item[$(iii)$] To a Boolean algebra $\left( B,\vee \wedge ,\partial
,0,1\right) $ we can associate a Boolean ring $\left( B,+,\cdot \right) ,$
where
\begin{eqnarray*}
x+y &=&\left( x\vee y\right) \wedge \partial \left( x\wedge y\right) , \\
x\cdot y &=&x\wedge y,
\end{eqnarray*}%
for all $x,y\in B.$ Conversely, if $\left( B,+,\cdot \right) $ is a Boolean
ring, we can associate a Boolean algebra $\left( B,\vee \wedge ,\partial
,0,1\right) ,$ where
\begin{eqnarray*}
x\vee y &=&x+y+xy, \\
x\wedge y &=&xy, \\
\partial x &=&1+x;
\end{eqnarray*}

\item[$(iv)$] Let $\left( I,+,\cdot \right) $ be an ideal in a Boolean ring $%
\left( B,+,\cdot \right) $, therefore $I$ is an ideal in the Boolean algebra
$\left( B,\vee \wedge ,\partial ,0,1\right) $. The converse is also
true.\medskip
\end{enumerate}
\end{remark}

\begin{remark}
\label{rem3.15}

\begin{enumerate}
\item[$(i)$]  If $X$ is an MV-algebra and $I$ is an ideal (prime ideal) in $X
$, therefore on the Wajsberg algebra structure, obtained as in Remark 3.7.
ii), we have that the same set $I$ is an ideal (prime ideal) in $X$ as
Wajsberg algebra. The converse is also true.

\item[$(ii)$] Finite MV-algebras of order $2^{t}$ are Boolean algebras.

\item[$(iii)$] Between ideals in a Boolean algebra and ideals in the
associated Boolean ring it is a bijective correspondence, that means, if $I$
is an ideal in a \ Boolean algebra, the same set $I,$ with the corresponded
multiplications, is an ideal in the associated Boolean ring. The converse is
also true.\medskip\
\end{enumerate}
\end{remark}

We consider $\left( R,+,\cdot \right) $ a finite, commutative, unitary ring
and $I,J$ be two ideals. Let $c_{I}$ and $c_{J}$ be the codewords attached
to these sets, as in Definition \ref{def3.3}.\medskip

\begin{proposition}
\label{prop3.16}\textbf{\ }\textit{With the above notations, we have that:}

\begin{enumerate}
\item[$(i)$] \textit{To the set} $I\Delta J$ \textit{correspond the codeword}
$c_{I}+c_{J}=c_{I}\oplus c_{J}$\textit{, where} $\oplus $ \textit{is the
XOR-operation;}

\item[$(ii)$] \textit{If} $I_{1},I_{2},...,I_{q}$ \textit{are ideals in the
ring} $R$ \textit{and} $c_{I_{1}},c_{I_{2}},...,c_{I_{q}}$ \textit{are the
attached codewords, therefore the vectors} $c_{I_{1}},c_{I_{2}},...,c_{I_{q}}
$ \textit{are linearly independent vectors. }
\end{enumerate}
\end{proposition}

\begin{proof}
$(i).$ It is clear, by straightforward computations.

$(ii).$ Let $R$ has $n$ elements. We work on the vector space $V=\underset{%
n-time}{\underbrace{\mathbb{Z}_{2}\times \mathbb{Z}_{2}\times ...\times
\mathbb{Z}_{2}}}$ over the field $\mathbb{Z}\,_{2}$. We consider $\alpha
_{1}c_{I_{1}}+...\alpha _{q}c_{I_{q}}=0,$where $\alpha _{1},...\alpha
\,_{q}\in \mathbb{Z}_{2}.$ Supposing that $\alpha _{1}=...=\alpha \,_{q}=1$,
we have that $\alpha _{1}c_{I_{1}}+...\alpha _{q}c_{I_{q}}=0$ implies that $%
I_{1}\Delta I_{2}\Delta ...\Delta I_{q}=\emptyset $. Without losing the
generality, since symmetric difference is associative, from here we have
that $I_{1}\Delta I_{2}\Delta ...\Delta I_{q-1}=I_{q}$, which is false,
since $I_{q}$ has an ideal structure and $I_{1}\Delta I_{2}\Delta ...\Delta
I_{q-1}$ is not an ideal, from Remark \ref{rem3.2}.
\end{proof}

With the above notations, we consider a matrix $M_{C},$ with rows the
codewords associated to the ideals $I_{1},I_{2},...,I_{q}$,%
\begin{equation*}
M_{C}=\left(
\begin{array}{c}
c_{I_{1}} \\
c_{I_{2}} \\
... \\
c_{I_{q}}%
\end{array}%
\right) .
\end{equation*}%
Since these rows are linearly independent vectors, the matrix $M_{C}$ can be
considered as a generating matrix for a code, called \textit{the code
associated to the ideals} $I_{1},I_{2},...,I_{q}$, denoted $\mathcal{C}%
_{I_{1}I_{2},...I_{q}}$.$_{{}}\medskip $

\begin{theorem}
\label{th3.17} \textit{Let} $\left( B,\vee \wedge ,\partial ,0,1\right) $
\textit{be a finite Boolean algebra of order }$2^{n}$. \textit{The following
statements are true:}

\begin{enumerate}
\item[$(i)$] \textit{The algebra} $B$ \textit{has} $n$ \textit{ideals of
order} $2^{n-1}$\textit{;}

\item[$(ii)$] \textit{The code associated to above ideals generate a
Hadamard code of the type} $[2^{n},n,2^{n-1}]_{2}$, $n\geq 2$.\medskip
\textit{\ }
\end{enumerate}
\end{theorem}

\begin{proof}
$(i).$ It is clear, since ideals in the Boolean algebra structure are ideals
in the associated Boolean ring and vice-versa.

$(ii).$ Let $I_{1},I_{2},...,I_{n}$ be the ideals of order $2^{n-1}$. With
the above notations, we consider a matrix $M_{C},$ with rows the codewords
associated to these ideals,%
\begin{equation*}
M_{C}=\left(
\begin{array}{c}
c_{I_{1}} \\
c_{I_{2}} \\
... \\
c_{I_{n}}%
\end{array}%
\right) .
\end{equation*}%
Due to the correspondence between the ideals in the Boolean algebra
structure, the ideals in the associated Boolean ring and Proposition \ref%
{prop3.16}, we have that the rows of the matrix $M_{C}$ are linearly
independent vectors. Since $I_{1},I_{2},...,I_{n}$ are the ideals of order $%
2^{n-1}$, the associated codewords have $2^{n-1}$ nonzero elements,
therefore the Hamming distance is $d_{H}=2^{n-1}$. From here, we have that $%
M_{C}$ is a generating matrix for the code $\mathcal{C}_{I_{1}I_{2},...I_{n}}
$, which is a Hadamard code of the type $[2^{n},n,2^{n-1}]_{2}$, $n\geq 2$%
.\medskip
\end{proof}

\begin{remark}
\label{rem3.18} A generating matrix $M_{C}$ of a Hadamard code $\mathcal{C}$
of the type $[2^{n},n,2^{n-1}]_{2}$, $n\geq 2$, has $2^{n-1}n$ elements
equal with $1$. If the matrix has the following form: on the row $i$ we have
the first $2^{n-i}$ elements equal to $1$, the next $2^{n-i}$ elements equal
to $0$, and so on, for $i\geq 1$, we call this form the \textit{Boolean form}
of the generating matrix of the Hadamard code $\mathcal{C}$ and we denote it
$M_{B}$.\medskip
\end{remark}

\begin{remark}
\label{rem3.19}

\begin{enumerate}
\item[$(i).$]   If $G$, a $r\times s$ matrix over a field $K$, is a
generating matrix for a linear code $\mathcal{C}$, then any matrix which is
row equivalent to $G$ is also a generating matrix for the code $\mathcal{C}$%
. Two row equivalent matrices of the same type have the same row space. The
row space of a matrix is the set of all possible linear combinations of its
row vectors, that means it is a vector subspace of the space $K^{s}$, with
dimension the rank of the matrix $G$, $rankG$. From here, we have that two
matrices are row equivalent if and only if one can be deduced to the other
by a sequence of elementary row operations.

\item[$(ii).$]  If $G$ is a generating matrix for a linear code $\mathcal{C}$%
, then, from the above notations, we have that $M_{C}$ and $M_{B}$ are row
equivalent, therefore these matrices generate the same Hadamard code $%
\mathcal{C}$ of the type $[2^{n},n,2^{n-1}]_{2}$, $n\geq 2$. \medskip
\end{enumerate}
\end{remark}

\begin{theorem}
\label{th3.20}\textbf{\ }\ \textit{With the above notations, let} \thinspace
$M_{B}$ \textit{be the Boolean form of a generating matrix of the Hadamard
code of the type} $[2^{n},n,2^{n-1}]_{2}$, $n\geq 2$.\ \textit{We can
construct a Boolean algebra} $\mathcal{B}$ \textit{of order} $2^{n}$ \textit{%
which has} $n$ \textit{ideals of order} $2^{n-1},$ \textit{with associated
codewords being the rows of a matrix }$M_{B}$.\medskip
\end{theorem}

\begin{proof}
\  We consider the set $B_{i}=\{0_{i},1_{i}\}$, with $0_{i}\leq _{i}1_{i}$, $%
i\in \{1,2,...,n\}$. On $B_{i}$ we define the following multiplication:$%
\begin{tabular}{l|ll}
$\circ _{i}$ & $0_{i}$ & $1_{i}$ \\ \hline
$0_{i}$ & $1_{i}$ & $1_{i}$ \\
$1_{i}$ & $0_{i}$ & $1_{i}$%
\end{tabular}%
.$

It is clear that $\left( B_{i},\circ _{i},^{\prime },1_{i}\right) $, where $%
0_{i}^{\prime }=1_{i}$ and $1_{i}^{\prime }=0_{i}$, is a Wajsberg algebra of
order $2$. On $B_{i}$ we have the following partial order relation $%
x_{i}\leq _{i}y_{i}~$if~and~only~if~$x_{i}\circ _{i}y_{i}=1_{i}.$

Therefore, on the Cartesian product $\mathcal{B}=B_{1}\times B_{2}\times
...\times B_{n}$ we define a component-wise multiplication, denoted $%
\diamond $. From here, we have that $\left( \mathcal{B},\diamond ,^{\prime },%
\mathbf{1}\right) $, where $\left( x_{1},x_{2},...,x_{n}\right) ^{\prime
}=\left( x_{1}^{\prime },x_{2}^{\prime },...,x_{n}^{\prime }\right) $ and $%
\mathbf{1=(}1,1,...,1)$, is a Wajsberg algebra of order $2^{n}$. We write
and denote the elements of $\mathcal{B}$ in the lexicographic order. The
element $\left( 0_{1},0_{2},...,0_{n}\right) $, denoted $\left(
0,0,...,0\right) $ or $\mathbf{0}$ it is the first element in $\mathcal{B}.$
With $\mathbf{1}$ we denote $\mathbf{(}1,1,...,1)=\left(
1_{1},1_{2},...,1_{n}\right) $ which is the last element in $\mathcal{B}$.
From Definition 3.8, on $\mathcal{B}$ we have the following partial order
relation\
\begin{equation*}
x\leq _{\mathcal{B}}y~\text{if~and~only~if~}x\diamond y=\mathbf{1}\text{.}
\end{equation*}%
It is clear that on $\mathcal{B}$ we have that $x\leq _{\mathcal{B}}y$ if
and only if $x_{i}\leq _{i}y_{i}$, for $i\in \{1,2,...,n\}$. From the
Wajsberg algebra structure we obtain the $MV$-algebra structure on $\mathcal{%
B}$, which is a Boolean algebra structure, with the multiplication $x\oplus
y=x^{\prime }\diamond y$ ($\oplus $ which is the component-wise XOR-sum).
The ideals of order $2^{n-1}$ in this Boolean algebra of order $2^{n}$ are
generated by the maximal elements in respect to the order relation $\leq _{%
\mathcal{B}}$. These elements have $n-1$ "nonzero" components. \ First
maximal element, in the lexicographic order, is $m_{1}=\left(
0,1,1,...,1\right) $. This element generates an ideal of order $2^{n-1}$,
containing all elements $x_{j}$ equal or less than $m_{1}$ in respect to the
order relation $\leq _{\mathcal{B}}$. \ Indeed, all these elements $x_{j}$
are maximum $n-2$ nonzero components and $x_{ji}\leq _{i}m_{1i}$, $i\in
\{1,2,...,n\},j\in \{1,2,...,2^{n-1}\}$, with the first component always
zero. We denote with $J_{1}$ the set all elements equal or less than $m_{1}$%
. It results that $J_{1}$ with the multiplication $\oplus $ is isomorphic to
the vector space $\mathbb{Z}_{2}^{n-1}$, therefore $J_{1}$ is an ideal in $%
\mathcal{B}$. The codeword corresponding to this ideal is $\left(
1,1,...,1,0,0,...,0\right) $ in which the first $2^{n-1}$ positions are
equal with $1$ and the next $2^{n-1}$ are $0$ and represent the first row of
the matrix $M_{\mathcal{B}}$. The next maximal element in lexicographic
order is $m_{2}=\left( 1,0,1,...,1\right) ,$ with zero on the second
position and $1$ in the rest. This element generates an ideal $J_{2}~$of
order $2^{n-1}$, containing all elements $x_{j}$ equal or less than $m_{2}$
in respect to the order relation $\leq _{\mathcal{B}}$. All these elements $%
x_{j}$ are maximum $n-2$ nonzero components and $x_{ji}\leq _{i}m_{1i}$, $%
i\in \{1,2,...,n\},j\in \{1,2,...,2^{n-1}\},$ with the second component
always zero. With the same reason as above, we have that $J_{2},$ with the
multiplication $\oplus ,$ is isomorphic to the vector space $\mathbb{Z}%
_{2}^{n-1}$, therefore $J_{2}$ is an ideal in $\mathcal{B}$. The codeword
corresponding to this ideal is $\left(
1,1,...,1,0,0,...,0,1,1,...,0,...\right) $, with the first $2^{n-2}$
positions equal with $1$, the next $2^{n-2}$ are $0$ and so on. This
codeword represent the second row of the matrix $M_{\mathcal{B}}$, etc.
\end{proof}

\begin{example}
\label{ex3.21} In \cite{FHSV}, the authors described all Wajsberg algebras
of order less or equal with $9$. In the following, we provide some examples
of codes associated to these algebras.

\textbf{Case} $n=4.$ We have two types of Wajsberg algebras of order $4$. \
First type is a totally ordered set which has no proper ideals and the
second type is a partially ordered Wajsberg algebra, $W=\{0,a,b,1\}.$ This
algebra has the multiplication given by the following table:
\begin{equation*}
\begin{tabular}{l|llll}
$\circ $ & $0$ & $a$ & $b$ & $1$ \\ \hline
$0$ & $1$ & $1$ & $1$ & $1$ \\
$a$ & $b$ & $1$ & $b$ & $1$ \\
$b$ & $a$ & $a$ & $1$ & $1$ \\
$1$ & $0$ & $a$ & $b$ & $1$%
\end{tabular}%
\text{.}
\end{equation*}

This algebra has two proper ideals $I=\{0,a\}$ and $J=\{0,b\}$. The
associated $MV$-algebra of this algebra is a Boolean algebra. We consider $%
c_{I}=\left( 1,1,0,0\right) $ and $c_{J}=\left( 1,0,1,0\right) $ the
codewords attached to the ideals $I$ and $J$. The matrix
\begin{equation*}
M_{C}=\left(
\begin{array}{cccc}
1 & 1 & 0 & 0 \\
1 & 0 & 1 & 0%
\end{array}%
\right)
\end{equation*}%
is the generating matrix for the Hadamard code of the type $\left(
2^{2},2,2\right) $. As in Remark \ref{rem3.4}, this matrix has as columns
all $2$-bits vectors over $\mathbb{Z}_{2}:\{11,10,01,00\}$.

\textbf{Case} $n=8$. We consider the partially ordered Wajsberg algebra, $%
W=\{0,a,b,c,d,e,f,1\}$ with the multiplication given by the following table:
\begin{equation*}
\begin{tabular}{l|llllllll}
$\circ $ & $0$ & $a$ & $b$ & $c$ & $d$ & $e$ & $f$ & $1$ \\ \hline
$0$ & $1$ & $1$ & $1$ & $1$ & $1$ & $1$ & $1$ & $1$ \\
$a$ & $f$ & $1$ & $f$ & $1$ & $f$ & $1$ & $f$ & $1$ \\
$b$ & $e$ & $e$ & $1$ & $1$ & $e$ & $e$ & $1$ & $1$ \\
$c$ & $d$ & $e$ & $f$ & $1$ & $d$ & $e$ & $f$ & $1$ \\
$d$ & $c$ & $c$ & $c$ & $c$ & $1$ & $1$ & $1$ & $1$ \\
$e$ & $b$ & $c$ & $b$ & $c$ & $f$ & $1$ & $f$ & $1$ \\
$f$ & $a$ & $a$ & $c$ & $c$ & $e$ & $e$ & $1$ & $1$ \\
$1$ & $0$ & $a$ & $b$ & $c$ & $d$ & $e$ & $f$ & $1$%
\end{tabular}%
.
\end{equation*}

All proper ideals are of the form $I_{1}=\{0,a\}$, $I_{2}=\{0,b\}$, $%
I_{3}=\{0,d\}$, $I_{4}=\{0,a,b,c\}$, $I_{5}=\{0,a,d,e\}$, $I_{6}=\{0,b,d,f\}$
are also prime ideals. This algebra has three ideals of order three $%
I_{4},I_{5},I_{6}$. The associated $MV$-algebra of this algebra is a Boolean
algebra. We consider $c_{I_{4}}=\left( 1,1,1,1,0,0,0,0\right) ,$ $%
c_{I_{5}}=\left( 1,1,0,0,1,1,0,0\right) ,c_{I_{6}}=\left(
1,0,1,0,1,0,1,0\right) $ the codewords attached to the ideals $%
I_{4},I_{5},I_{6}$. The matrix
\begin{equation*}
M_{C}=\left(
\begin{array}{cccccccc}
1 & 1 & 1 & 1 & 0 & 0 & 0 & 0 \\
1 & 1 & 0 & 0 & 1 & 1 & 0 & 0 \\
1 & 0 & 1 & 0 & 1 & 0 & 1 & 0%
\end{array}%
\right)
\end{equation*}%
is the generating matrix for the Hadamard code $\left( 2^{3},2,2^{2}\right) $%
. As in Remark \ref{rem3.4}, this matrix has as columns all $3$-bits vectors
over $\mathbb{Z}_{2}$, namely $\{111,110,101,100,011,010,001,000\}$.

\end{example}

\begin{remark}
\label{rem3.22} (\cite{FHSV}, case $n=9$) If a finite Wajsberg algebra has
an even number of proper ideals, we can consider their associated codewords,
as above. The obtained generating matrix generate a linear code with Hamming
distance $\geq 3$. Indeed, for $n=9$, we consider the partially ordered
Wajsberg algebra, $W=\{0,a,b,c,d,e,f,g,1\}$ with the multiplication given by
the following table:
\begin{equation*}
\begin{tabular}{l|lllllllll}
$\circ $ & $0$ & $a$ & $b$ & $c$ & $d$ & $e$ & $f$ & $g$ & $1$ \\ \hline
$0$ & $1$ & $1$ & $1$ & $1$ & $1$ & $1$ & $1$ & $1$ & $1$ \\
$a$ & $g$ & $1$ & $1$ & $g$ & $1$ & $1$ & $g$ & $1$ & $1$ \\
$b$ & $f$ & $g$ & $1$ & $f$ & $g$ & $1$ & $f$ & $g$ & $1$ \\
$c$ & $e$ & $e$ & $e$ & $1$ & $1$ & $1$ & $1$ & $1$ & $1$ \\
$d$ & $d$ & $e$ & $e$ & $g$ & $1$ & $1$ & $g$ & $1$ & $1$ \\
$e$ & $c$ & $d$ & $e$ & $f$ & $g$ & $1$ & $f$ & $g$ & $1$ \\
$f$ & $b$ & $a$ & $b$ & $e$ & $e$ & $e$ & $1$ & $1$ & $1$ \\
$g$ & $a$ & $b$ & $b$ & $d$ & $e$ & $e$ & $g$ & $1$ & $1$ \\
$1$ & $0$ & $a$ & $b$ & $c$ & $d$ & $e$ & $f$ & $g$ & $1$%
\end{tabular}%
.
\end{equation*}%
All proper ideals are $I_{1}=\{0,a,b\}$, $I_{2}=\{0,c,f\}$ and are also
prime ideals. We consider $c_{I_{1}}=\left( 1,1,1,0,0,0,0,0,0\right) $ and $%
c_{I_{2}}=\left( 1,0,0,1,0,0,1,0,0\right) ,$ the codewords attached to the
ideals $I_{1},I_{2}$. The matrix
\begin{equation*}
M_{C}=\left(
\begin{array}{ccccccccc}
1 & 1 & 1 & 0 & 0 & 0 & 0 & 0 & 0 \\
1 & 0 & 0 & 1 & 0 & 0 & 1 & 0 & 0%
\end{array}%
\right)
\end{equation*}%
is the generating matrix for the linear code of the form $\left(
9,2,3\right) ,\mathcal{C}_{I_{1}I_{2}}$. The even numbers of ideals assure
us that the rows in the generating matrix are linear independent vectors.
\end{remark}

\section{Conclusions}

In this paper, based on ideals,  we investigate residuated lattices from
fuzzy set theory and lattice theory point of view. Also we found connections
between the fuzzy sets associated to ideals in a Boolean algebras and
Hadamard codes. As a further research, we will study other connections
between fuzzy sets and some type of algebras of logic.

\medskip Cristina Flaut

{\small Faculty of Mathematics and Computer Science, Ovidius University,}

{\small Bd. Mamaia 124, 900527, Constan\c{t}a, Rom\^{a}nia,}

{\small \ http://www.univ-ovidius.ro/math/}

{\small e-mail: cflaut@univ-ovidius.ro; cristina\_flaut@yahoo.com}

\bigskip

Dana Piciu

{\small Faculty of \ Science, University of Craiova, }

{\small A.I. Cuza Street, 13, 200585, Craiova, Rom\^{a}nia,}

{\small http://www.math.ucv.ro/dep\_mate/}

{\small e-mail: dana.piciu@edu.ucv.ro, piciudanamarina@yahoo.com}

\bigskip

Bianca Liana Bercea

{\small PhD student at Doctoral School of Mathematics,}

{\small Ovidius University of Constan\c{t}a, Rom\^{a}nia}

{\small e-mail: biancaliana99@yahoo.com}

\end{document}